\documentclass[12pt]{amsart}

\usepackage[margin=1.15in]{geometry}

\usepackage[T2A]{fontenc}
\usepackage[utf8]{inputenc}
\usepackage[english]{babel}

\usepackage{amsmath,amscd,amssymb,amsfonts,latexsym}
\usepackage{wasysym}
\usepackage{mathrsfs}
\usepackage{mathtools,hhline}
\usepackage{color}
\usepackage{bm}
\usepackage[all, cmtip]{xy}

\usepackage{url}
\usepackage{mathtools}
\usepackage{amsmath}

\definecolor{hot}{RGB}{65,105,225}

\usepackage[pagebackref=true,colorlinks=true, linkcolor=hot ,  citecolor=hot, urlcolor=hot]{hyperref}

\theoremstyle{plain}
\newtheorem{theorem}{Theorem}[section]

\newtheorem{corollary}[theorem]{Corollary}

\newtheorem{lemma}[theorem]{Lemma}

\theoremstyle{definition}
\newtheorem{definition}[theorem]{\sc Definition}
\newtheorem{example}[theorem]{\sc Example}
\newtheorem{remark}[theorem]{\sc Remark}

\numberwithin{equation}{section}

\newcommand\hot{\mathrm{h.o.t.}}

\newcommand\cB{\mathcal{B}}
\newcommand\cE{\mathcal{E}}

\newcommand\bx{\mathbf{x}}
\newcommand\ity{\infty}

\def\bR{\mathbb{R}}
\def\bC{\mathbb{C}}
\def\bP{\mathbb{P}}

\def\m{\setminus}

\newcommand{\NCone}{\mathscr{N}\mathrm{Cone}}

\DeclareMathOperator{\Sing}{Sing}                                

\DeclareMathOperator{\mult}{mult}
\DeclareMathOperator{\ord}{ord} 
 
  \DeclareMathOperator{\gen}{gen}              
\DeclareMathOperator{\reg}{reg} 
                
\DeclareMathOperator{\sing}{sing}
\DeclareMathOperator{\atyp}{atyp}
\DeclareMathOperator{\iflex}{iflex}
\DeclareMathOperator{\focal}{focal}
  \DeclareMathOperator{\Cone}{Cone}

\DeclareMathOperator{\EDdeg}{EDdeg}
\DeclareMathOperator{\ED}{ED}
\DeclareMathOperator{\tED}{\ED*}
\DeclareMathOperator{\Eu}{Eu} 
\DeclareMathOperator{\cl}{closure}

\title[ED discriminants]{Euclidean distance discriminants and Morse attractors}

\author{Cezar Joi\c ta}
\address{Institute of Mathematics of the Romanian Academy, P.O. Box 1-764,
 014700 Bucharest, Romania and Laboratoire Europ\' een Associ\'e  CNRS Franco-Roumain Math-Mode}
\email{Cezar.Joita@imar.ro}
\author{Dirk Siersma}
\address{Institute of Mathematics, Utrecht University, PO
Box 80010, \ 3508 TA Utrecht, The Netherlands.}
\email{D.Siersma@uu.nl}

\author{Mihai Tib\u ar}
\address{Univ. Lille, CNRS, UMR 8524 -- Laboratoire Paul Painlev\'e, F-59000 Lille,
France}  
\email{mihai-marius.tibar@univ-lille.fr}

\thanks{The authors acknowledge support from the project ``Singularities and Applications'' - CF 132/31.07.2023 funded by the European Union - NextGenerationEU - through Romania's National Recovery and Resilience Plan, and support by the grant CNRS-INSMI-IEA-329. }


\begin{document}


\begin{abstract}  
Our study concerns the Euclidean distance function in case of  complex plane curves. We decompose the ED discriminant into components which are responsible for three types of behavior of the Morse points.
Besides the traditional focal component, which is non--linear; the other components are lines. In particular we shed light on the ``atypical discriminant'' which is due to the loss of Morse critical  points at isotropic points at infinity. This phenomena is specific  for the complex  setting.

We find formulas for the number of Morse singularities which abut to the corresponding  type of  attractors when moving the centre of the  distance function toward a  point of the discriminant.

\end{abstract}

\maketitle



\section{Introduction}
 
 Early studies dedicated to the Euclidean distance emerged before 2000, with much older roots going back to the 19th century geometers.  For instance, if one considers the particular case of a curve  $X \subset \bR^2$ given by a real equation $f(x,y) = 0$,
the aim is to study the critical points of the Euclidean distance function:
\[D_u(x,y) = (x - u_{1})^{2} + (y - u_{2})^{2} \]
 from a  centre $u :=(u_1,u_2)$ to the variety $X$. In the case that $X$ is compact and smooth, $D_{u}$ is generically a Morse function, and the values  $u$ where $D_{u}$ has degenerate critical points are called \emph{discriminant}, or \emph{caustic}, or \emph{evolute}.
 These objects have been studied intensively in the past, see e.g. the recent study \cite{PRS} with its multiple references including to Huygens in the 17th century,  and to the ancient Greek geometer Apollonius.
 
On each connected component of the complement of the caustic, the number of Morse critical points of $D_u$ and their index is constant. 
Assuming now that $(x,y)$ are complex coordinates, the number of those complex critical points is known as the \emph{ED degree}, and it provides upper bounds for the real setting. The corresponding discriminant is called the \emph{ED discriminant}. These notions have been introduced in \cite{DHOST}, and have been studied in many papers ever since, see e.g.  \cite{Hor2017}, \cite{DGS}, \cite{Ho}. They  have applications to computer vision e.g. \cite{PST2017}, numerical algebraic geometry, data science, and other optimization problems e.g.  \cite{HS2014}, \cite{NRS2010}.

The earlier paper \cite{CT} contains a study of the ED discriminant under a different name, with a particular definition and within a restricted class of (projective) varieties.  

 From the topological side, more involved computation of $\EDdeg(X)$  have been done in \cite{MRW2018}, \cite{MRW5} etc,  in terms of the Morse formula from \cite{STV} for the \emph{global Euler obstruction}  $\Eu(X)$, and  in terms of vanishing cycles of a linear Morsification of a distance function where the data point is on the ED discriminant.
In particular the same authors have proved  in \cite{MRW2018} the \emph{multiview  conjecture} which had been stated in \cite{DHOST}. 
 
This type of study based on Morsifications appears to be extendable  to singular polynomial functions, see \cite{MT1}, \cite{MT2}. The most recent paper \cite{MT3} treats for the first time the case of Morse points disappearing at infinity, via a new principle of computation based on relative polar curves.

In this paper we consider the discriminant in the case of complex plane curves $X$, where several general striking phenomena already manifest.
In particular,  the "loss of Morse points at infinity`` has a central place in our study. This phenomenon shows that the bifurcation locus encoded by the discriminant may be partly due to the non-properness of the projection $\pi_2: \cE_X \to \bC^n$, see Definition \ref{d:incidence}. It occurs even in simple examples, and it is specific to the complex setting.\\ 

We will show that the discriminant can be decomposed in the following way:
  $$ \Delta_{\tED} = \Delta^{\focal} \cup \Delta^{\iflex}  \cup \Delta^{\atyp} \cup \Delta^{\sing}  $$
  
  $(1).$    \emph{The focal set} $\Delta^{\focal}$ represents the traditional evolute and is birational with the regular part of $X$. It is generically curved (may contain components of higher degree); while the other components are lines and related to certain isolated points on $X$ (or its completion  at infinity).

 The line components, that will be discussed are:
 
 $(2).$  \emph{The atypical discriminant} $\Delta^{\atyp}$, due to the Morse points which are ``lost'' at infinity. A necessary condition is that  $\overline{X}$ passes through an isotropic point. See \S\ref{s:atyp}.

 $(3).$  \emph{The singular  discriminant} $\Delta^{\sing}$, due to the Morse points which move to singularities of $X$. See \S\ref{s:sing}.
 
  $(4)$ \emph{The iflex discriminant } $\Delta^{\iflex}$, due to collision of Morse points on $X_{\reg}$, which move to flex points on $X$, with isotropic tangent lines (for short:\emph{ iflex points}). See \S\ref{s:reg}. 
  
   \emph{The regular discriminant} $\Delta^{\reg}$ is the union $\Delta^{\focal} \cup \Delta^{\iflex}$ and is just the discriminant set on $X_{\reg}$.
   
   These discriminants may intersect,  and may also have common components, which should then be lines.
   Several examples at the end will illustrate these notions and other phenomena, see \S\ref{s:examples}.

 \noindent The contents of our study are as follows.

In \S\ref{s:discrim} we recall two definitions of ED discriminants that one usually uses, the total ED discriminant $\Delta_{\tED}(X)$, and the strict ED discriminant $\Delta_{\ED}(X)$. We explain the first step of a classification for low ED degree,  equal to 0 and to 1. In \S\ref{ss:attract} we discuss the concept of attractor, as limit of Morse points, along a curves in $u$-space. 
We discuss the different types of discriminants in  \S\ref{s:atyp}, \S\ref{s:sing} and \S\ref{s:reg}. For each type of complex discriminant, we compute in \S\ref{ss:morseinfty}, \S\ref{s:sing} and \S\ref{ss:morsereg}, the number of Morse singularities which abut to attractors of Morse points, respectively.

We introduce also the concept of \emph{focal point} on $\Delta_p^{\sing}$ and $\Delta_{\xi}^{\atyp}$ and show  that these points are limits of $\Delta^{\focal}$  on each of them. Note that focal points on  $\Delta_p^{\sing}$ can occur on every point of $\Delta_p^{\sing}$, including $p$.
Several quite simple examples at \S\ref{s:examples} illustrate all these results and phenomena,  with detailed computations.

Let us remark   some  differences with the real setting. The two discriminants $\Delta^{\atyp}$ and $\Delta^{\iflex}$ don't occur; so no lines occur because of them. Therefore no attractors at infinity. The discriminant $\Delta^{\sing}$ can occur in the real setting.

\medskip
 The authors thank the hospitality of IMAR (the Simion Stoilow Institute of Mathematics) in Bucharest, where part an important part of the work has been carried out.
 
We hereby share the sad news that Mihai Tib\u ar, our colleague and author of this article, passed away in September 2025. He was actively involved in the writing and revision of this article.

\section{ED degree and ED discriminant}\label{s:discrim}

\subsection{Two definitions of the ED discriminant}

We consider  an algebraic curve $X\subset \bC^{2}$, with reduced structure. Its singular set $\Sing X$ consists of a finite subset of points.

\begin{definition}\label{d:defgood}
The \emph{ED degree of $X$}, denoted by $\EDdeg(X)$, is the number of Morse (critical)  points $p\in X_{\reg}$ of a generic distance function $D_{u}$, and this number is independent of the choice of the generic centre $u$ in a Zariski-open subset of $\bC^{2}$.
\end{definition}
\begin{definition}\label{d:total-dcr}
The \emph{total ED discriminant}  $\Delta_{\tED}(X)$ is  the set of points $u \in \bC^{2}$ such that  the function
$D_{u}$ has less than  $\EDdeg(X)$ Morse  points on $X_{\reg}$, or that $D_{u}$ is not a Morse function on $X_{\reg}$.\footnote{In particular $u\in\Delta_{\tED}(X)$ if $D_{u}$ has non-isolated singularities.}
   \end{definition}
 
Note that by definition $\Delta_{\tED}(X)$ is a closed set, as the complement of an open set. Indeed, if $D_{u}$ a Morse function
 then $D_{u'}$ is a Morse function for any $u'$ in some neighbourhood of $u$.
 
 \
 
 A different definition was given in \cite{DHOST}, as follows.
Consider the following incidence variety, a variant of the conormal of $X$,  where $\bx = (x,y)$ and $(u-\bx)$ is viewed  as a 1-form:
$$  \cE_X := \cl \bigl\{ (\bx,u)\in X_{\reg}\times \bC^{2} \mid   \ (u-\bx)|T_{\bx}X_{\reg}=0  \bigr\}  \subset X\times \bC^{2}  \subset \bC^{2}\times \bC^{2},$$
and let us remark that  $\dim \cE_X = 2$.  $\cE_X $ is just the closure of the set of critcal points of $D_u$. Let $\pi_{1} : \cE_X  \to X$ and $\pi_{2} : \cE_X  \to \bC^{2}$ be the  projections on the first and second factor, respectively.
The projection $\pi_{2}$ is generically finite, and the degree of this finite map is the \emph{ED degree of $X$}, like also defined above at Definition \ref{d:defgood}.

\begin{definition}\label{d:incidence}
The branch locus of $\pi_{2}$ is called  \emph{the (strict) ED discriminant}, and will be denoted here by $\Delta_{\ED}(X)$.
The branch locus is defined as the set of $u \in \bC^2$ where $\pi_2^{-1}(u)$ has cardinality different from $\EDdeg(X)$ (points are counted without multiplicities).
\end{definition}

Note that $\Delta_{\tED}(X)$ is related to the bifurcation of the function, while $\Delta_{\ED}(X)$ is related to the critical points.
By the above  definitions, we have the inclusion $\Delta_{\ED}(X)\subset \Delta_{\tED}(X)$, which may not be an equality, see e.g. Examples \ref{ss:lines} and \ref{ss:cusp}. In fact: 
 $\Delta_{\tED}(X) =\Delta_{\ED}(X) \cup \Delta^{\sing}$.

\subsection{Terminology and two simple examples}\label{e:2ex}\   
We say that a line in $\bC^2$ is \emph{isotropic} if its projective closure passes though one of the (isotropic) points $[1; \pm i]$ on the projective line $H^{\infty}$ at infinity. Notation $Q^{\infty} =  [1;i] \cup [i;1]$. We say that a line $K$
 is \emph{normal} to a line $L$ at some point $p\in L$ if $\langle q-p, r-p \rangle$ is equal to 0 for any $q\in K$ and any $r\in L$ and where $(q-p)$ is viewed  as a 1-form. Note that an isotropic line is normal to itself.

 
\begin{example}[Lines] \label{ss:lines}\
 Lines in $\bC^{2}$ do not have all the same ED degree, see Theorem \ref{t:lines}(a-b). Let $X$ be the union of two non-isotropic lines intersecting at a point $p$.  By additivity the ED degree is then $\EDdeg(X) =2$. According to the definitions,  the ED discriminant   $\Delta_{\tED}(X)$ contains the two normal lines at $p$, whereas $\Delta_{\ED}(X)$ is empty.
\end{example}

\begin{example}[Cusp]\label{ss:cusp}\ 
The plane cusp $X:= \{ (x,y) \in \bC^2 \mid x^{3}=y^{2}\}$ has  $\EDdeg(X)= 4$.  The ED discriminant $\Delta_{\ED}(X)$ is a smooth curve of degree 4  passing through the origin.
If $u\in \Delta_{\ED}(X)$ is a point different from the origin, then the distance function $D_{u}$ has precisely one non-Morse critical point on $X_{\reg}$ produced by the merging of two of the Morse points.

The origin is a special point of $\Delta_{\ED}(X)$: the distance function from the origin, denoted by $D_{0}$, has only two Morse points on $X_{\reg}$ while two other Morse points had merged in the origin.

We have $\Delta_{\tED}(X) = \Delta_{\ED}(X)\cup  L$, where       $L =\{x=0\}$, the line normal to the tangent to the cusp. At some point $u\in L$ different from the origin, 
the distance function $D_{u}$ has only 3 Morse points on $X_{\reg}$ while the 4th Morse point had merged with  the singular point of $X$. 

\end{example}

\subsection{First step of a classification}\label{ss:classif}

\begin{theorem}\label{t:lines}
 Let  $X\subset \bC^{2}$ be an irreducible reduced curve. Then
 \begin{enumerate}
 \item $\EDdeg(X) =0$ $\Longleftrightarrow$ $X$ is an isotropic line.
  In this case $\Delta_{\tED}(X)=X$.

 \item $\EDdeg(X) =1$ $\Longleftrightarrow$ $X$ is a line different from an isotropic line. In this case $\Delta_{\ED}(X)$ is empty.
 
 \item  The discriminant $\Delta_{\ED}(X)$ contains some point $u= (u_1, u_2) \in \bC^2$ such that $\dim \pi_{2}^{-1}(u)>0$ if and only if:
 
 (i).   either $X = \{ (x, y)\in \bC^2 \mid (x-u_{1})^{2}+ (y-u_{2})^{2} = \alpha\}$ for a certain $\alpha  \in \bC^{*}$. 
 
 (ii).  or   $X$ is an isotropic line. 
   \end{enumerate}  
\end{theorem}

\begin{proof}
In (a) and (b) the implications ``$\Leftarrow$'' are both clear by straightforward computation;  we will therefore show ``$\Rightarrow$'' only.  

\noindent (a). 
 $\EDdeg(X) =0$ implies that the normal to the tangent space $T_{p}X_{\reg}$ is this space itself. If $T_{p}X_{\reg} = \bC\langle(a,b)\rangle$, then the only vectors $(a,b)$ which have this property are those verifying the equation $a^{2}+b^{2} =0$. This means that for any $p\in X_{\reg}$, one has either $T_{p}X_{\reg} = \bC\langle(x, ix)\rangle$
or $T_{p}X_{\reg} = \bC\langle(x, -ix)\rangle$. This implies that $X_{\reg}$ is one of the lines $\{ x \pm iy = \alpha\}$, for some $\alpha\in \bC$.\\

 \noindent  (b).
 Assume X is not a line. Let $p$ be a generic point on $X_{\reg}$ and $L_p$ the unique normal at $p$. $D_u$ has the  Morse   critical point $p$  for  generic points $u \in L_p$. Let $L_q$ be any other normal contaning $Q \in X_{\reg}$ as a Morse point then since $ED=1$ it follows that $L_q$ and $L_p$ don't intersect and are therefore parallel for almost all $q$. This implies that $X_{\reg}$ is contained in a line, which is not isotropic since $\EDdeg(X) \ne 0$.\\

\noindent  (c).  The hypothesis implies that for some point $u\in \Delta_{\ED}(X)$, the function $D_{u}$ has non-isolated singularity on $X$. Since this is necessarily contained in a single level of $D_{u}$, it follows that $X$ contains $\{ (x-u_{1})^{2}+ (y^{2}-u_{2})^{2} = \alpha\}$ for some $\alpha\in \bC$, and since $X$ is irreducible, the twofold conclusion follows.
\end{proof}


 \subsection{The attractors of Morse points}\label{ss:attract}
 Let $X\subset \bC^{2}$ be a reduced affine variety of dimension 1.
 Let $\hat{u}\in \Delta_{\tED}$. A  point $p\in \overline{X}$ will be called \emph{attractor} for  $D_{\hat{u}}$ if there are   Morse points of $D_{u(s)}$ which abut to $p$ when $s\to 0$, where $u(s)$ is some path at $u(0):=\hat{u}$. 
 The attractors fall into three types,
  
 (1).  One or both points of $\overline{X} \cap Q^{\ity}$ may be attractors, as shown in Lemma \ref{l:atyp}.  See \S\ref{ss:morseinfty} for all details, and  Examples \ref{ex:morseind} and \ref{e:isotropic}.
 
(2).  Any $p\in \Sing X$ is an attractor, since  at least one Morse point of $D_{u(s)}$ abuts to it. See \S\ref{s:sing} for all details.
 
(3). Points $p\in X_{\reg}$ to which more than one Morse singularities of $D_{u(s)}$ abut. Such a point appears to be 
 a non-Morse singularity of $\Sing D_{\hat{u}}$,  and it varies with $\hat{u}\in \Delta_{\tED}$. See \S\ref{ss:reg}, and  Example \ref{ex:regatyp}.

\section{The atypical discriminant}\label{s:atyp}

We define the discriminant $\Delta^{\atyp}$ as the subset of $\Delta_{\ED}(X)$ which is due to the loss of Morse points to infinity, and we find its structure. As before, let $H^{\infty}$ be the line at infinity. 
\begin{definition}\label{d:atyp}
Let $\overline X$ denote the closure of $X$ in $\bP^2$ and $X^{\infty} :=\overline X\cap H^\ity$.
For a point $\xi\in X^{\infty}$, let $\Gamma$ be a local  branch of  $\overline X$ at $\xi$. 

 We denote by $\Delta^{\atyp}_{\xi}(\Gamma)\subset\Delta_{\ED}(X)$  the set of all points $u\in \bC^2$ such that there are a sequence
 $\{u_n\}_{n\geq 1}\subset \bC^2$ with $u_n\to u$, and a sequence
$\{\bx_n\}_{n\geq 1}\subset (\Gamma\m H^\ity)$ with $\bx_n\to\xi$, such that $(u_{n}-\bx_{n})|T_{\bx}X_{\reg}=0$. 
The \emph{atypical discriminant} is then defined as follows:
 $$\Delta^{\atyp} :=\bigcup_{\Gamma,\xi}  \Delta^{\atyp}_{\xi}(\Gamma)$$
  where the union runs over all local branches $\Gamma$ of $\overline X$ at all points $\xi\in X^{\infty}$. If the choosen branch is clear we write $\Delta^{\atyp}(X)$ .
  \end{definition}

\subsection{The structure of $\Delta^{\atyp}$}\label{s:struct}\ \\
Let $\gamma:B\to \Gamma$ be a local holomorphic  parametrization of $\Gamma$ at $\xi$, where 
$B$ is some disk in $\bC$ centred at $0$ of small enough radius, and $\gamma(0)=\xi$. If $x$ and $y$ denote the coordinates 
of $\bC^2$, then for $t\in B$, we write $x(t)=x(\gamma(t))$ and $y(t)=y(\gamma(t))$.  It follows that the functions $x(t)$ and $y(t)$ are meromorphic on $B$ and holomorphic on $B\setminus{0}$. We thus may write them on some small enough disk $B'\subset B\subset \bC$ centred at the origin,  as follows:
$$x(t)=\frac{P(t)}{t^k}, \ y(t)=\frac{Q(t)}{t^k},$$
where $P(t)$ and $Q(t)$ are holomorphic, and $P(0)$ and $Q(0)$ are not both equal to zero. See also Lemma \ref{l:atyp}
 for the significance of the exponent $k$.
\medskip

Under these notations,  we have $\xi =[P(0);Q(0)]\in H^{\ity}$. For $t\in B\m\{0\}$ and $u=(u_1,u_2)\in\bC^2$,  we have:
$\bigl( (x(t),y(t)),u\bigr)\in \cE_X$ if and only if 
$$\frac{(tP'(t)-kP(t))}{t^{k+1}}\Big(\frac{P(t)}{t^k}-u_1\Big) + \frac{(tQ'(t)-kQ(t))}{t^{k+1}}\Big(\frac{Q(t)}{t^k}-u_2\Big)=0.$$

This yields a holomorphic function  $h:B\times\bC^2\to \bC$ defined as:
$$h(t,u)=\bigl(tP'(t)-kP(t)\bigr)(P(t)-u_1t^k) + \bigl(tQ'(t)-kQ(t)\bigr)\bigl(Q(t)-u_2t^k\bigr) $$ 
which is linear in the coordinates $(u_1,u_2)$.

For $t\in B\m\{0\}$ and $u\in\bC^2$, we then obtain the equivalence: 
\begin{equation}\label{eq:normal}
 \bigl( (x(t),y(t)),u\bigr)\in \cE_X  \Longleftrightarrow h(t,u)=0.
\end{equation}

\medskip

If we write  $h(t,u)=\sum_{j\geq 0} h_j(u)t^j$, then we have:
\begin{lemma}
\begin{enumerate}
\item For any $j\leq k-1$,  $h_j(u)=h_j\in\bC$, for all $u\in \bC^{2}$ and  \\ $h_0 = -k (P(0)^2 + Q(0)^2)$,

\item The function $h_k(u)$ is of the form $h_k(u)=kP(0)u_1 + kQ(0)u_2+\text{constant}$. 
Since $P(0)$ and $Q(0)$ are not both zero  by our assumption, it also follows that the function $h_k(u)$ is not constant;

\item For any $i>k$, the function $h_i(u)$ is a  (possibly constant) linear function.
\end{enumerate}\end{lemma}

Let us point out the geometric interpretation of the integer $k$, and the role of the isotropic points at infinity:
\begin{lemma} \label{l:atyp}\ 
Let  $\xi \in X^{\ity}$ and  let $\Gamma$ be a branch of $\overline{X}$ at $\xi$. Then:
\begin{enumerate}
\item  $k = \mult_{\xi}(\Gamma, H^{\ity})$.
\item  Let $Q^\ity := \{x^{2} + y^{2} =0\} \subset H^\ity$.
If $\xi \not\in X^{\ity}\cap Q^\ity=\emptyset$ then $\Delta^{\atyp}_{\xi}(\Gamma)=\emptyset$. 
\end{enumerate} 
\end{lemma}
\begin{proof}
\noindent (a). Since $P(0)$ and $Q(0)$ are not both zero, let us assume  that $P(0) \not= 0$. In coordinates at $\xi\in H^{\ity}\subset \bP^{2}$ we then have $z=\frac1x$ and $w = \frac{y}{x}$.  Composing with the parametrization of $\Gamma$ we get $z(t) = \frac{1}{x(t)} = t^{k}r(t)$ where $r$ is holomorphic and $r(0) \not= 0$.  We therefore get:
\begin{equation}\label{eq:PQ}
 \mult_{\xi}(\Gamma, H^{\ity}) = \ord_{0} z(t) = k,
\end{equation}
and observe this holds also in the other case $Q(0) \not= 0$.

\noindent  (b).
If $\xi \not\in X^\ity\cap Q^\ity$ then, for any branch $\Gamma$ of $\overline{X}$ at $\xi$, we have
$P(0)^2+Q(0)^2\neq 0$, hence $h_0\neq 0$. This shows that the equation $h(t,u)=0$ has no solutions in a small enough neighbourhood of $\xi$.  
\end{proof}

\begin{theorem} \label{t:atyp} \
Let $\xi\in X^{\ity}\cap Q^\ity$, and let $\Gamma$ be a branch of $\overline{X}$ at $\xi$. Then:
\begin{enumerate}
\item $u\in \Delta^{\atyp}_{\xi}(\Gamma)$ if and only if $\ord_{t}h(t,u) \ge 1+ \mult_{\xi}(\Gamma, H^{\ity})$.
\item   If $\Delta^{\atyp}_{\xi}(\Gamma)\neq\emptyset$, then $\Delta^{\atyp}_{\xi}(\Gamma)$ is the line $\{u\in \bC^{2} \mid h_k(u)=0\}$. 

In particular,  $\Delta^{\atyp}$ is a finite union of isotropic lines.
\end{enumerate}
\end{theorem}

\begin{proof}
\noindent (a). \sloppy
We have to show that 
 $u\in \Delta^{\atyp}_{\xi}(\Gamma)$ if and only if $h_0=\cdots = h_{k-1}=0$ in $h(t,u)$, and $h_{k}(u) =0$.
If $h_0,\ldots, h_{k-1}$ are not all equal to $0$, then let $0\leq j_1\leq k-1$ be the first index such that  $h_{j_1}\neq 0$.
We then have: 
$$h(t,u)=t^{j_1}\Big(h_{j_1}+\sum_{j>j_1}h_j(u)t^{j-j_1}\Big).$$ 
Let $K$ be a compact subset of $\bC^2$ containing a neighborhood of some point $\hat{u}\in \Delta^{\atyp}_{\xi}(\Gamma)$.
   Then, since $(t,u)\to \sum_{j>j_1}h_j(u)t^{j-j_1}$ is holomorphic,  we get $\lim_{t\to 0} \sum_{j>j_1}h_j(u)t^{j-j_1}= 0$ uniformly for $u\in K$.
This implies that $h(t,u)\neq 0$, for $|t|\neq 0$ small enough, and for all $u\in K$, which contradicts the assumption  that  $\hat{u}\in \Delta^{\atyp}_{\xi}(\Gamma)$. We conclude that  $\Delta^{\atyp}_{\xi}(\Gamma)=\emptyset$. The continuation and the reciprocal will be proved in (b).

\medskip

\noindent (b). Let us assume now that $h_0=\cdots =h_{k-1}=0$. We then write $h(t,u)=t^k\widetilde h(t,u)$ where 
\begin{equation}\label{eq:morseinfty}
 \widetilde h(t,u)=h_k(u)+\sum_{j>k}h_j(u)t^{j-k}.
\end{equation}
We have to show that $u\in \Delta^{\atyp}_{\xi}(\Gamma)$ if and only if $h_k(u)=0$.
\medskip

``$\Rightarrow$'': If $h_k(u)\neq 0$, then a similar argument as at (a) applied to $\widetilde h(t,u)$ shows that $u\not\in \Delta^{\atyp}_{\xi}(\Gamma)$.
\medskip

``$\Leftarrow$'': Let $h_k(u_{1}, u_{2})=0$. We have to show that for every neighborhood $V$ of $u$ and every disk $D \subset B \subset \bC$  centred at the origin, there exist $v\in V$ and $t\in D\m\{0\}$ such that $\widetilde h(t,v)=0$.
Suppose that this is not the case. Denoting by $Z(\widetilde h)$ the zero-set of $\widetilde h$, we would  then have
$$\big(Z(\widetilde h)\cap (D\times V)\big)\subset \{0\} \times V.$$
We also have the equality $Z(\widetilde h)\cap (\{0\} \times V)=\{0\} \times Z(h_k)$. It would follow the inclusion:
\begin{equation}\label{eq:inclZ}
  \big(Z(\widetilde h)\cap (D\times V)\big)\subset \{0\} \times Z(h_k).
\end{equation}

The set  $\{0\} \times Z(h_k)$ has dimension at most 1, while $Z(\widetilde h)\cap (D\times V)$ has dimension 2 since it cannot be empty, as $\widetilde h(u,0)=0$. We obtain in this way a contradiction to the inclusion \eqref{eq:inclZ}.

This shows in particular that $\Delta^{\atyp}_{\xi}(\Gamma)$ is an isotropic line which contains the point $\xi$ in its closure at infinity.
We finally note that $\Delta^{\atyp}$ is the union of $\Delta^{\atyp}_{\xi}(\Gamma)$ over all branches at infinity of $\overline{X}$, thus $\Delta^{\atyp}$ is a union of  isotropic lines.
\end{proof}

\begin{corollary}\label{c2}
Let $\Gamma$ be a branch of $\overline X$ at $\xi \in X^\ity\cap Q^{\ity}$. 

Then $\Delta^{\atyp}(\Gamma)\neq\emptyset$ if and only if   $\Gamma$  is not tangent  at $\xi$ to  the line at infinity $H^\ity$.
\end{corollary}

\begin{proof}
 Let us assume $\xi = [i;1]$, since a similar proof works for the other point of $Q^{\ity}$. 
  Let  $(w, z)$ be local coordinates of $\bP^2$ at $\xi$, such that $H^\ity=\{z=0\}$ and we have:
$$\ x=\frac{w}{z}, \ y=\frac{1}{z}.$$
 Our hypothesis ``$H^\ity$ is not tangent to $\Gamma$ at $\xi$'' implies that we may choose a parametrization for $\Gamma$ at $\xi$ of the form $z(t)=t^k$, $w(t)=i+t^kP_1(t)$,    where $P_1$ is a holomorphic function on a neighborhood of the origin, and where $\ord_0 z(t) = k =  \mult_{\xi}(\Gamma, H^{\ity})\ge 1$, as shown in \eqref{eq:PQ}.

Under the preceding notations,  we have $Q(t)\equiv 1$, $P(t)=i+t^k P_1(t)$, and we get 
\begin{align*}
h(t,u)&=\bigl(t^kP_1'(t)-ki\bigr)\bigl(i+t^kP_1(t)-u_1t^k\bigr))-k+ku_2t^k\\
&=t^k\Big[t P_1'(t)\bigl(i+t^kP_1(t)-u_1t^k\bigr)-kiP_1(t)+kiu_1+ku_2\Big]
\end{align*}

By Theorem \ref{t:atyp}(a),   $u \in \Delta^{\atyp}(\Gamma)\neq\emptyset$ if and only if $\ord_t h(t,u) \ge 1+k$.  From the above expression of $h(t,u)$ we deduce:  $\ord_t h(t,u) \ge 1+k$ $\Longleftrightarrow$ $iu_1+u_2 + C=0$, where $C = iP_1'(0) - iP_1(0)$ is a constant. This is the equation of an isotropic line. We deduce that $\Delta_{\xi}^{\atyp}(\Gamma)$ is this line, and therefore it is not empty.

\

Reciprocally, let us assume now that $\Gamma$ is tangent to  $H^\ity$ at $\xi$. By Lemma \ref{l:atyp}(a), this implies $k\ge 2$.  
A parametrization for $\Gamma$ is of the form $z(t)=t^k$, $w(t)=i+\sum_{j\geq r}a_jt^j$, where $1\le r<k$.  

As before,  we have $Q(t)\equiv 1$ and $P(t)=i+a_rt^r+\hot$  The expansion of $h(t,u)$ looks then as follows:
\begin{align*}
h(t,u)&=\bigl(tP'(t)-kP(t)\bigr)(P(t)-u_1t^k) + \bigl(tQ'(t)-kQ(t)\bigr)\bigl(Q(t)-u_2t^k\bigr) \\
&= k ( -i r_0 + i u_1 - q_0 u_2) + (1-k)[ i r_1 - q_1 u_2]  t\\  
& \; \;  \; \; \; \; \; + (2-k) [ i r_2 -q_2 u_2] t^2 + (3-k)[i r_3 -q_3 u_2] t^3 + \hot
\end{align*}
We have  $a_r\not= 0$, $r-2k\neq 0$ since $r<k$, thus $\ord_t h(t,u) < k$. Then  Theorem \ref{t:atyp}(a) tells that  $\Delta^{\atyp}_{\xi}(\Gamma)=\emptyset$. 
\end{proof}

\subsection{Morse numbers at infinity}  \label{ss:morseinfty}
We have shown in \S\ref{s:struct} that $\Delta^{\atyp}$ is a union of lines.
Our purpose is now to fix a point $\xi \in \overline{X}\cap Q^{\ity}$ and find the number of Morse singularities of $D_{u}$
which abut to it when the centre $u$  moves  from outside $\Delta^{\atyp}$ toward some $\hat{u}\in \Delta^{\atyp}$.  We will in fact do much more than that. 

Let $\Gamma$ be a local  branch of  $\overline X$ at $\xi$. We assume that $\hat{u}\in \Delta^{\atyp}_{\xi}(\Gamma)\subset \Delta^{\atyp}$,  as defined in  \S\ref{s:struct}. 

We will now prove the formula for the number of  Morse points which are lost at infinity.  

\begin{theorem}\label{t:morseinfty}
  Let $\hat{u}\in \Delta^{\atyp}_{\xi}(\Gamma)\not= \emptyset$. Let $\cB\in \bC $ denote a small disk centred at the origin, and let $u: \cB \to \bC^{2}$ be a continuous path such that $u(0) = \hat{u}$, and that  $h_k(u(s)) \not= 0$ for all $s\not=0$.
  
   Then the number of  Morse points of $D_{u(s)}$, 
which abut to $\xi$ along $\Gamma$ when $s\to 0$ is: 
 
\begin{equation}\label{eq:morseinfty}
  m_{\Gamma}(\hat{u}) := \ord_{0}\Bigl(\sum_{j>k}h_j(\hat{u})t^{j}\Bigr) - \mult_{\xi}(\Gamma, H^{\ity})
\end{equation}
if $\ord_{0}\sum_{j>k}h_j(\hat{u})t^{j}$ is finite.  In this case, the integer $m_{\Gamma}(\hat{u}) >0$ is independent of the choice of the path $u(s)$ at $\hat{u}$.  
\end{theorem}
\begin{remark}\label{r:ordinfty}
  The excepted case in Theorem \ref{t:morseinfty} is in fact a very special curve $X$. Indeed,   the order $\ord_{0}\sum_{j>k}h_j(\hat{u})t^{j}$ is infinite if and only if the series is identically zero,  and this is equivalent to $X =\ {  (x-\hat{u}_1})^2 + (y-\hat{u}_2)^2= \alpha \} $, for some $\alpha\in \bC$.
\end{remark}

\begin{proof}[Proof of Theorem \ref{t:morseinfty}]
We  will use Theorem \ref{t:atyp}, its preliminaries and its proof with all notations and details.
Replacing $u$ by $u(s)$ in \eqref{eq:morseinfty} yields:
$$\widetilde h(t,u(s)) =h_k(u(s))+\sum_{j>k}h_j(u(s))t^{j-k}.$$
Note that by our choice of the path $u(s)$ we have that $h_k(u(s)) \not= 0$ for all $s\not=0$ close enough to 0.

The number of Morse points which abut to $\xi$ is precisely the number of solutions in variable $t$ of the equation 
$\widetilde h(t,u(s))=0$ which converge to 0 when $s\to 0$. 
This is precisely equal to $\ord_{t}\sum_{j>k}h_j(\hat{u})t^{j-k}$, and we remind that $k = \mult_{\xi}(\Gamma, H^{\ity})$ by Lemma \ref{l:atyp}(a).  In particular, this result is independent of the choice of the path $u(s)$. 
Since we have assumed that $\Delta^{\atyp}_{\xi}(\Gamma)\not= \emptyset$, there must exist Morse singularities which abut at $\xi = \Gamma\cap Q^\ity$, thus $m_{\Gamma}(\hat{u}) >0$.
\end{proof}

\begin{remark}\label{r:morseinfty}
For $j\ge k$, the set  $L_{j}:= \{ h_{j}(u) =0 \}$ is a line if $h_{j}(u) \not\equiv 0$.
 The number of Morse points $m_{\Gamma}(u)$ interprets as $j-k$, where  $j>k$ is  the first index such that $L_{j}$ is a nonempty line, different from $L_k$.
 
  Since $L_{j}\cap L_{k}$ consists of at most one point, it follows that the number $m_{\Gamma}(\hat{u})$   is constant for all points $\hat{u}\in L_{k}$, except possibly at the single point $\tilde{u} = L_{j}\cap L_{k}$, for which the Morse number $m_{\Gamma}(\tilde{u})$ takes a  higher value, or  $\widetilde h(t,\tilde{u}) \equiv 0$, in which case $D_{\tilde{u}}$ has non-isolated singularities, see Remark \ref{r:ordinfty}. The generic number will be denoted by $m^{\gen}_{\Gamma}$. See Examples \ref{ex:morseind} and \ref{e:isotropic}. The single point $\tilde{u}$, where $m_{\Gamma}(u)$ jumps we will call the \textit{`focal'} point on $\Delta^{\atyp}_{\xi}$.   
  \end{remark}

\begin{definition}\label{d:morseinfty}
 We call $m^{\gen}_{\Gamma}$ the  \emph{generic Morse number at $\xi$ along $\Gamma$}. The number:
 $$ m^{\gen}_{\xi} := \sum_{\Gamma} m^{\gen}_{\Gamma}
 $$
 will be called the  \emph{generic Morse number at $\xi$}.
 
   We say that $m_{\Gamma}(\hat u) > m^{\gen}_{\Gamma}$ is the \emph{exceptional Morse number at $\xi$ along $\Gamma$}, whenever this
 point  exists and the associated number is finite.  We will call  $\hat u\in L_k$ the focal point on $\Delta^{\atyp}_{\xi}$.  
 \end{definition}

See Example \ref{ex:morseind} where $\hat u$ exists but the number $m_{\Gamma}(\hat u)$ is not defined because the order is infinite (cf Theorem \ref{t:morseinfty}).

 \section{Structure of $\Delta^{\sing}$, and Morse numbers at the attractors of  $\Sing X$.}\label{s:sing}

  We consider $X$ with reduced structure; consequently $X$ has at most isolated singularities.   We recall from 
  the Introduction  that $\Delta^{\sing}$ is the subset of points  $u\in\Delta_{\tED}(X)$ such that, when $u(s) \to u(0) = u $,  at least a Morse point of $D_{u(s)}$ abuts to a singularity of $X$.

 \begin{theorem}\label{t:morseaffine}
Let $p\in \Sing X$. Then $p$ is an attractor, and: 
\begin{enumerate}
 \item The singular discriminant is the union 
 $$\Delta^{\sing} = \cup_{p\in \Sing X}\NCone_{p}X ,$$ 
 where $\NCone_{p}X$ denotes the union of all the normal lines at $p$ to the tangent cone $\Cone_{p}X$.
 \item  Let $\hat{u} \in N_p\Gamma$, the normal to $\Gamma$. The number of  Morse points of $D_{u(s)}$, 
which abut to $p$ along $\Gamma$ when $s\to 0$, is: 
   $$m_{\Gamma}(\hat{u}) :=  1 - \mult_{p}\Gamma + \ord_{t}\sum_{j\ge \mult_{p}\Gamma }h_j(\hat{u})t^{j}$$
 This number is independent of the choice of the path $u(s)$.
\end{enumerate}
\end{theorem}

\begin{proof}
\noindent
(a).  Let us consider a local branch $\Gamma$ of $X$ at $p\in \Sing X$. Let $\gamma: B \to \Gamma$ be a  parametrization of $\Gamma$ at $p$:
  $$x(t) :=x(\gamma(t)) = p_{1} + c_1 t^{\alpha} + \hot  \mbox{ and } y(t) :=y(\gamma(t))= p_{2} + c_2 t^{\beta} +\hot,$$
where where $B$ denotes a small enough disk in $\bC$ centred at $0$, $c_1 \not=0$ and $c_2 \not=0$ . We can suppose  $\beta \ge \alpha = \mult_{p}\Gamma>1$.
 
 We then have the equivalence:  
\begin{equation}\label{eq:singeq}
 \bigl( ((x(t),y(t)),u\bigr)\in \cE_X \Longleftrightarrow  x'(t)(x(t) -u_{1}) + y'(t)(y(t) -u_{2})  = 0.
\end{equation}
  
This equivalence  \eqref{eq:singeq} means that the number of Morse points  which abut to $p$ when $t\to 0$ is 
precisely the maximal number of solutions in $t$ of the equation:
\begin{equation}\label{eq:Morsesol}
h(t,u) :=   x'(t)\bigl( x(t) -u_{1}\bigr) + y'(t)\bigl( y(t) -u_{2}\bigr)  = 0
\end{equation}
which converge to 0 when $s\to 0$.

We have:
\[  x'(t)\bigl( x(t) -u_{1}\bigr) =c_1  \alpha (p_{1}-u_{1})t^{\alpha-1} + c_1 \alpha t^{2\alpha-1} +\hot,\]
\[  y'(t)\bigl( y(t) -u_{2}\bigr) = c_2\beta (p_{2}-u_{2})t^{\beta-1} + c_2\beta t^{2\beta-1} + \hot ,\]

We write $h(t,u)= \sum_{j\ge 0}h_j(u)t^{j}$ as a holomorphic function of variable $t$ with coefficients depending on $u$.
For any $j\ge 0$,  the coefficient $h_j(u)$ is a linear function in $u_{1}, u_{2}$, possibly identically zero. 
Let  $h(t,u) = t^{\alpha-1}\widetilde h(t,u)$, where $\widetilde h(t,u) := \sum_{j\ge \alpha-1}h_j(u)t^{j-\alpha+1}$.

Note that the constant term in $\widetilde h(t,u) $ as a series in variable $t$ is: \\
 $h_{\alpha-1}(u) = c_1\alpha(p_{1}-u_{1})$ if $\alpha <\beta$, or \\
 $h_{\alpha-1}(u) = c_1\alpha(p_{1}-u_{1}) + c_2\beta (p_{2}-u_{2})$ if $\alpha =\beta$.

 It follows that either the line $L :=\{ u_{1} - p_{1}=0\}$, or the line $L :=\{ c_1\alpha(p_{1}-u_{1}) + c_2\beta (p_{2}-u_{2})=0\}$, respectively, 
  is included in $\Delta_p^{\sing}$.  Note that this line is the normal at $p$ to the tangent cone of the branch $\Gamma$ in the coordinate system that we have set in the beginning.  This proves the point (a) of our statement.\\
  
\noindent
(b).  Let us fix now a point $\hat{u}$ on a line $L\in \NCone_{p}X$. In order to compute how many Morse points abut to $p$ along $\Gamma$ when approaching $\hat{u}$, let us consider a small disk $B\in \bC $ centred at the origin and a continuous path  $u: B \to \bC^{2}$ such that $u(0) = \hat{u}$, and that  $h_{\alpha-1}(u(s)) \not= 0$ for all $s\not=0$.

Let  $B\in \bC $ denote a small disk centred at the origin, and let $u: B \to \bC^{2}$ be some continuous path such that $u(0) = \hat{u}$, and that  $h_k(u(s)) \not= 0$ for all $s\not=0$.

 Replacing $u$ by $u(s)$ yields:
$$\widetilde h(t,u(s)) =  \sum_{j\ge \alpha-1}h_j(u(s))t^{j-\alpha +1}.$$
The number of Morse points which abut to $p$ is then the number of solutions in variable $t$ of the equation 
$\widetilde h(t,u(s))=0$ which converge to 0 when $s\to 0$.
This is precisely equal to $\ord_{t}\sum_{j\ge \alpha}h_j(\hat{u})t^{j-\alpha+1}$. In particular, this number is independent of the choice of the path $u(s)$.
\end{proof}

\begin{remark} 
Note that case (b) of Theorem \ref{t:morseaffine} avoids the case that several branches meet at $p \in \Sing X$. If the point $\hat{u}$ lies on a single normal line at $p$ the formula still applies. If $\hat{u} = p$ then can use the contributions for each branch separately, with a correction related to the branching.  
In case one of the branches is smooth it can occur that $m_{p, \Gamma}(u) =  0$  (compare Remark \ref{r:cusptypes}).

\end{remark}

\begin{remark}\label{r:exceptionalpoint}
 There is a generic Morse number  for all $u\in L$, except possibly a unique exceptional point $\tilde u$ of $L$ for which the number is higher, or $D_{\tilde u}$ has a non-isolated singularity.  
In Example \ref{ss:cusp}, we have $X:= \{x^{3}=y^{2}\}$, and then $\Delta^{\sing} = \NCone_{(0,0)}X$ is the axis $u_{1}=0$. The generic Morse number 
$m_{(0,0), X}(u)$ is 1, and the exceptional point  of this line is $\tilde u=(0,0)$, where the Morse number is 2.
The exceptional point on $\Delta_p^{\sing}$ is called \textit{focal} point in analogy to the the regular case.
\end{remark}

\begin{remark}\label{r:cusptypes}
We investigate next the behavior along $\Delta_p^{\sing} = \{u_1=0\}$ in more detail in case that the tangent line is non-isotropic. For convenience we assume $p = (0,0)$. We will use now the parametrization:\\
  $x(t) :=x(\gamma(t)) =  t^{\alpha}  \mbox{ and } y(t) :=y(\gamma(t))=  c t^{\beta} +\hot,$
  with $\alpha < \beta$.\\
so: $D_u(t)= (u_1-t^{\alpha})^2 + (u_2 -c \, t^{\beta} +\hot )^2$, ($c \ne 0$). 
It follows that
$$ h(t,u) = - \alpha u_1 t^{\alpha -1}+  \alpha t^{2\alpha -1} - c \beta u_2 t^{\beta-1} + \hot $$
We restrict to $u_1=0$ and get:
$$ h(t,u) = \alpha t^{2\alpha -1} - c \beta u_2 t^{\beta-1} + \hot $$
We distinguish 3 cases (illustrated in Figure \ref{fig:3cusps}):
\begin{itemize}
\item[$\beta > 2\alpha$]:  $t^{2\alpha -1 }$ is dominant; $m_{p, \Gamma}(u)= \alpha$ for all $u \in  L$. There is no finite  focal point;
\item[$\beta < 2\alpha$]: $-u_2 t^{2\beta -1 }$ is dominant. If $u_2 \ne 0$:  $m_{p, \Gamma}(u)= \beta -\alpha$ and if  $\tilde{u}_2 = 0$ then $m_{p, \Gamma}(\tilde{u})= \alpha$. There is exactly one focal point, the point $p \in \Sing X$;
\item[$\beta=2\alpha$]: $\alpha (1- 2 cu_2)t^{2\alpha -1 }$ is dominant. If $2 c  u_2\ne 1$ then $m_{p, \Gamma}(u)= \alpha $. So the focal point is $(0, 1/2c)$. For the focal point  $m_{p, \Gamma}(\tilde{u})> \alpha$ or perhaps infinity.  With generic higher order terms $m_{p, \Gamma}(\tilde{u}) = \alpha+ 1$.
\end{itemize}
If follows from this that  $m_{p, \Gamma}(u) \ge 1$  and   $m_{p, \Gamma}(u) = 1$  happens only in thecase: $ \beta = \alpha +1$. If  $m_{p, \Gamma}(u)= 0$ the the branch $\Gamma$ must be smooth in $p$.

\begin{figure}[hbt!]\label{fig:3cusps}
 \includegraphics[width=120mm]{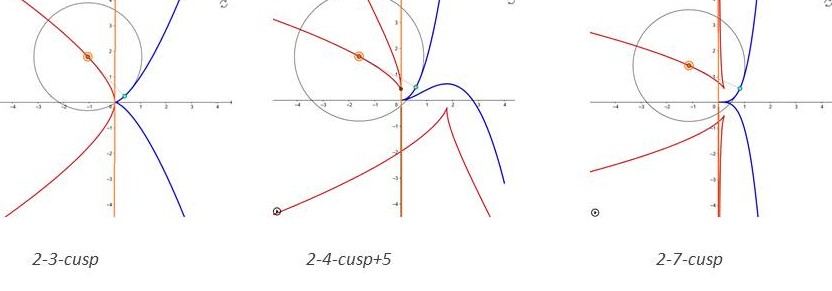}
 \caption{Some real pictures of cusps, together with discriminants. The first and the third picture are genuine cups, the second has  generic terms of degree 5. In all 3 cases the vertical line $\Delta_p ^{\sing}$ is the normal at $p$.  The focal sets $\Delta^{\focal}$ behave different: the 2-3-cusp intersects  $\Delta_p ^{\sing}$  in $p$, the 2-7 cusp goes to infinity. The 2-4-cusp+5 has a finite intersection with $\Delta_p^{\sing}$. Note that the other cusps on the focal set are not related to $p$, but to `extrema of curvature' of $X$ outside $p$. }
\end{figure}

An easy oberservation shows that the focal points belong to the closure of the focal locus $\Delta^{\focal}$  which is defined in the next subsection.  It also follows that $\Delta^{\sing}$ is not contained in $\Delta_{\ED}$. For a non-isotropic tangent we have:

\begin{corollary}
The discriminant $\Delta^{\focal}$ intersects $\Delta_p^{\sing}$  in the focal point if $\beta \le 2 \alpha$ and `at infinity' if $\beta > 2 \alpha$.
\end{corollary} These intersections is related  to the local study near the point $p$. Other intersections $L \cap \Delta^{\reg}$, which come from two different points on the curve  may occur (see Example \ref{e:isotropic}).
A similar computation can be done in case of an isotropic tangent line in $p$  with the conclusion that $\Delta^{\focal}$ intersects $\Delta_p^{\sing}$ in $p \in X$. See Example \ref{e:isotropic}, where isotropic coordinates are used. 
\end{remark}

 \section{The regular discriminant $\Delta^{\reg}$, and Morse attractors on $X_{\reg}$.}\label{s:reg}

 \subsection{A special role for flex points}
The study of $D_u$ on the regular part of X over the complex numbers resembles the real theory of \emph{focal sets}. In the classical real setting, this is known as an \emph{evolute}, or a \emph{caustic}, see e.g. \cite{Mi}, \cite{Tr},\cite{CT}. 
But in the complex case we will find new components at certain flex points.
We  define $\Delta^{\reg}$ as the closure of the subset of points $u\in\Delta_{\ED}(X)$ such that $D_u$ has a degenerate critical point on $X_{\reg}$. Such a singularity is an attractor for at least 2 Morse points of some generic small deformation of $D_u$.

Let $(x(t),y(t))$ be a local parametrization of $X$ at some point $p := (x(t_0),y(t_0))\in X_{\reg}$.
\begin{definition}\label{d:flex}
 A point $p= (x(t_0),y(t_0))\in X_{\reg}$ is called a \emph{flex point} if $x'(t_0)y''(t_0)-y'(t_0)x''(t_0) = 0$. 
 We say that the tangent line to $X_{\reg}$ at $p$ is \emph{isotropic} if it verifies the condition $(x'(t_0))^2+(y'(t_0))^2=0$.
 
 These two definitions do not depend on the chosen parametrization.
\end{definition}

\begin{theorem}\label{t:reg}
Let $p=(p_1,p_2)\in X_{\reg}$ and $(x(t),y(t))$ be a local parametrization for $X$ at $p$.  
\begin{enumerate}
\item  If $p$ is not a flex point,  then there exists a unique $u\in\bC^2$ such that $p$ is an attractor for $D_u$.
\item  If $p$ is a flex point, and  if the tangent to $X$ at $p$ is not isotropic, then $p$ is not an attractor for $D_u$, $\forall u\in\Delta_{\ED}(X)$.

\item  If $p$ is a flex point,  and  if the tangent to $X$ at $p$ is isotropic,  then $p$ is an attractor for $D_u$ for every point $u$ on the line tangent to $X$ at $p$.
\end{enumerate}
\end{theorem}

\begin{proof}
We recall that:
$$h(t,u) :=   x'(t)\bigl( x(t) -u_{1}\bigr) + y'(t)\bigl( y(t) -u_{2}\bigr). $$

Let $t_0$ be the point in the domain of  the parametrization $(x(t),y(t))$ such that $(x(t_0),y(t_0))=p$.
Without less of generality we can assum $t_o=0$.
The Taylor series at $t_0$ are:
\[  
\begin{cases}
 x(t)= p_1+x'(t_0)(t-t_0)+ \frac{x''(t_0)}{2} \cdot (t-t_0)^2 +\hot\\
y(t)= p_2+y'(t_0)(t)+ \frac{y''(t_0)}{2} \cdot (t-t_0)^2 +\hot
\end{cases}
\]
and therefore:
\begin{align*}
h(t,u)=&\bigl[ x'(t_0)(p_1-u_1)+y'(t_0)(p_2-u_2)\bigr] +
\\
&\bigl[ (x'(t_0))^2+x''(t_0)(p_1-u_1)+(y'(t_0))^2+y''(t_0)(p_2-u_2)\bigr] (t-t_0)+ \hot
\end{align*}

 We write (again)  $h(t,u) = h_0(u) + h_1(u) t + h_2(u) t^2 + \hot$
	
	In order to determine the multiplicities of zero (i.e. the singularity type of $D_u$), we consider the equations $h_k(u)=0$, rewritten as a system of linear equations:
\[  
\begin{cases}
x'(t_0)(p_1-u_1)+y'(t_0)(p_2-u_2)=0\\
 x''(t_0)(p_1-u_1)+y''(t_0)(p_2-u_2)=-(x'(t_0))^2-(y'(t_0))^2 \\
\end{cases}
\]

The point $p\in X_{\reg}$ is an attractor (for at least 2 Morse points) if and only if  $\ord_{t_{0}}h(t,u) >1$, thus if and only if $u=(u_{1}, u_{2})$ is a solution of the
linear system.	
If its determinant $D = x'(t_0)y''(t_0)-y'(t_0)x''(t_0)$ is not 0, then the system  
has a unique solution\footnote{This corresponds in the real geometry to the familiar fact that for every non-flex point there is a unique focal centre on the normal line to $X$ through $p$.}. If $D =0$, then the system has solutions if and only if $(x'(t_0))^2+(y'(t_0))^2=0$, and in this case the set of solutions is the line passing through $p$, normal to $X$, and thus tangent to $X$ because it is isotropic.
\end{proof}

\subsection{Structure of $\Delta^{\reg}$}\label{ss:reg}\ \\
 As we have seen, $\Delta^{\reg}$ may have line components due to Theorem \ref{t:reg}(b), see also Example \ref{ex:regatyp}.  We also have:
\begin{corollary}\label{c:reg}
 If $X$ is irreducible and is not a line, then $\Delta^{\reg}$ has a unique component which is not a line.
 \end{corollary}
\begin{proof}
 Since the non-flex points are dense in $X$ whenever $X$ is not a line, let  $p\in X_{\reg}$ be a non-flex point,  and we consider a local parametrization $(x(t),y(t))$ at $p:= (x(0),y(0))$. Then the system:
\[
\begin{cases}
x'(t)(x(t)-u_1)+y'(t)(y(t)-u_2)=0\\
 x''(t)(x(t)-u_1)+y''(t)(y(t)-u_2)=-(x'(t))^2-(y'(t))^2
\end{cases}
\]
 has a unique solution $u(t)=(u_{1}(t), u_{2}(t))$ for $t$ close enough to 0. We therefore obtain 
a parametrization for the germ of $\Delta^{\reg}$ at $\tilde u = u(0)$, exactly like in the classical real setting (\cite{PRS}):
 We use the notations $I(t) = (x'(t))^2+(y'(t))^2$ and $D(t) = x'(t)y''(t)-y'(t)x''(t)$:
%
 \[\begin{cases}
u_1(t)=x(t)- y'(t) \frac{I(t)}{D(t)}\\
u_2(t)=y(t)+x'(t) \frac{I(t)}{D(t)}
\end{cases}\]

This germ of $\Delta^{\reg}$ at $p$ cannot be a line. Indeed, by taking the derivative with respect to $t$ of the first equation of the system, we get:
$$ x''(t)(x(t)-u_1(t))+(x'(t))^2-x'(t)u'_1(t)+y''(t)(y(t)-u_2(t))+(y'(t))^2-y'(t)u'_2(t)=0.$$
By using the second equation of the system we deduce: 
$$x'(t)u'_1(t)+y'(t)u'_2(t)=0.$$
The germ $\Delta^{\reg}$ at $u(0)$ is a line if and only if $u'_1(t)/u'_2(t)$ is constant for all $t$ close enough to 0, which by the above equation is equivalent to $x'(t)/y'(t)=$ const. This implies that $X$ is a line at $p$, thus it is an affine line, contradicting our assumption.
\end{proof}

\begin{remark}
The linear components are isotropic tangent lines at flex points with isotropic tangents (as mentioned in (c) of Theorem \ref{t:reg}). They will be denoted by $\Delta^{\iflex}_p \subset \Delta^{\reg}$. The unique component, which is not a line, mentioned in Corollary \ref{c:reg} is the complex version of the focal set and will be denoted by  $\Delta^{focal}$ .
 We have:
 \[  \Delta^{\reg}        =   \Delta^{\focal}    \cup  \bigcup_{p}   \Delta^{\iflex}     \]
\end{remark}
\begin{remark}
 Next we will discuss the limit points of $\Delta^{\focal}$ if $t \to t_0$ for flex points $p \in X_{\reg}$.  
 We will usue the above evolute formula and  follow the cases, mentioned in Theorem \ref{t:reg}: 
\begin{itemize}
\item[(a)] non-flex point: $u(t)$ is continuous and the limit $u(t_o)$ is finite.
\item[(b)] flex-point with non-isotropic tangent: ($D(t_0)= 0 $ and $I(t_0) \ne 0$): $u(t) \to \infty$ asymptotic to the normal at $p$.
\item[(c)] flex-point with isotropic tangent: ($D(t_0)= 0 $ and $I(t_0) = 0$):  the limit of $u(t) \to p$ is $p \in X$.

\end{itemize}
(c) needs some explanation, since the quotient $I(t) / D(t)$ could become undetermined. We mention first two technical lemma's; the first is obvious.

\begin{lemma}
Suppose that $w:D\to\bC$ is a holomorphic function defined on a neighborhood $D$ of $0\in\bC$. If $u(0)=0$ then $\frac{w}{w'}$ is holomorphic on a neighborhood $D'\subset D$ of $0$
and $\frac{w}{w'}(0)=0$.
\end{lemma}

\begin{lemma}\label{l:limit}
Suppose that $f,g:D\to\bC$ are holomorphic functions defined on a neighborhood $D$ of $0\in\bC$ such that $f(0)g(0)\neq 0$ and $f(0)^2+g(0)^2=0$. \\ Then 
$\frac{f^2+g^2}{f'g-g'f}$ is holomorphic on a neighborhood $D'\subset D$ of $0$
and $\frac{f^2+g^2}{f'g-g'f}(0)=0$.
\end{lemma}
\begin{proof}
Let $w=\frac{f}{g}+\frac{g}{f}=\frac{f^2+g^2}{fg}$. As we assumed that $f(0)g(0)\neq 0$, it follows that $w$ is holomorphic on a neighborhood of the origin.

At the same time $w'=(f'g-g'f)\frac{f^2-g^2}{f^2g^2}$. By the previous Lemma $\frac{w}{w'}$ is holomorphic on a neighborhood of the origin and its value there is $0$.
We have that $\frac{w}{w'}=\frac{f^2+g^2}{f'g-g'f}\frac{fg}{f^2-g^2}$.
it remains to notice that because $f(0)g(0)\neq 0$ and $f(0)^2+g(0)^2=0$ we have that $f^2(0)-g^2(0)\neq 0$.
\end{proof}

To show (c)  we use Lemma \ref{l:limit} for $f=x'$ and $g=y'$. This implies  $I(t) / D(t)$ is holomorphic with limit $0$.

 \end{remark}
\subsection{Morse numbers at attractors on $X_{\reg}$.}\label{ss:morsereg}
 Let $\hat{u}\in \Delta^{\reg}(X)$ and let $p\in \Sing {D_{\hat{u}}}_{|X}\cap X_{\reg}$. We call  \emph{Morse number at $p\in X_{\reg}$}, and denote it by $m_{p}$,  the number of Morse points which  abut to $p$ as $s\to 0$ in a Morse deformation $D_{u(s)}$ with $u(0)=\hat{u}$. 
 
 A point $p\in \Sing {D_{\hat{u}}}_{|X_{\reg}}$ is an \emph{attractor} if  $m_{p}\ge 2$, see \S\ref{ss:attract} point (3). An attractor is therefore  a singularity of ${D_{\hat{u}}}_{|X_{\reg}}$ at $p$ which is not Morse.

\begin{theorem}[Morse number at an attractor on $X_{\reg}$] \ \\
 The Morse number at $p\in \Sing {D_{\hat{u}}}_{|X_{\reg}}$ is:
 \[ m_{p}(u) = \mult_{p}\bigl( X, \{D_{u}= D_{u}(p)\}\bigr) -  1.
 \]
  \end{theorem}
  
\begin{proof}
This is a consequence of general classical results, as follows. The Milnor number of a holomorphic function germ  $f: (X,p)\to (\bC,0)$ with isolated singularity at a smooth point $p\in X_{\reg}$ is equal to the number of Morse points in some Morsification $f_{s}$ which abut to $p$ when $s\to 0$, cf Brieskorn \cite{Bri}, and see also \cite{Ti3} for a more general statement.

On the other hand the Milnor number of $f$ at $p\in X_{\reg}$,  in case $\dim_{p}X=1$, is equal to the multiplicity of $f$ at $p$ minus 1. In our case, the function $f$ is the restriction to $X$ of the Euclidean distance function $D_{\hat{u}}$, and therefore this multiplicity equals the intersection multiplicity $\mult_{p}\bigl( X, \{D_{\hat{u}}= D_{\hat{u}}(p)\}\bigr)$.
\end{proof}

\section{Examples}\label{s:examples}

\subsection {Examples in complex coordinates}
\begin{example}[The ``complex circle'']\label{ex:morseind}\ 

Let $X := \{x^{2} + y^{2} = 1\}\subset \bC^{2}$, and $D_{u} := (x-u_{1})^{2} + (y-u_{2})^{2}$.\\
We have $X^{\ity}\cap Q^{\ity} = Q^{\ity} = \{[1;i],  [i;1]\}$, and $\EDdeg(X)= 2$.

A parametrization of the unique branch of $X$ at $[1;i]$, which we will denote by $X_{[1;i]}$,   is $\gamma: x= \frac{1+t^{2}}{2t}, \  y= \frac{(1-t^{2})i}{2t}$,  for $s\to 0$.  We get, in the notations of \S\ref{s:struct}:  $k=1$, $P(t) = \frac{1+t^{2}}{2}$, $P'(t) = t$, and $Q(t) = \frac{(1-t^{2})i}{2}$, $Q'(t) = it$. After all simplifications, we obtain:
\[ h(t,u) = (u_{1} +i u_{2})t + (-u_{1} +i u_{2})t^{3} \]
which yields $\widetilde{h}(t,u) = (u_{1} +i u_{2}) + (-u_{1} +i u_{2})t^{2}$.

This shows that $\Delta_{[1;i]}^{\atyp}(X) = \{ u_{1} +i u_{2} =0\}$, and that $m^{\gen}_{X_{[1;i]}} =2$ in the notations of Remark \ref{r:morseinfty}, which means that there are 2 Morse points which abut to $[1;i]$.
The exceptional point on the line is $(0,0)$, for which we get $\widetilde{h}(t,u) \equiv 0$, which means that $\dim \Sing D_{(0,0)} >0$, in other words $D_{(0,0)}$ has non-isolated singularities on $X$.

The study at the other point at infinity $[i;1]$  is similar. By the symmetry, we get:
 $\Delta_{[i;1]}^{\atyp}(X) = \{ u_{1} - i u_{2} =0\}$ and $m^{\gen}_{X_{[i;1]}} =2$, with the same exceptional point $(0,0)$.
 
		We get $\Delta^{\atyp}(X) = \Delta_{[1;i]}^{\atyp}(X)\cup \Delta_{[i;1]}^{\atyp}(X) = \{u_{1}^{2}+u_{2}^{2}=0\}$, and
we actually have:
$$\Delta_{\tED}(X)= \Delta_{\ED}(X)= \Delta^{\atyp}(X).$$
Note that for $u=(0,0)$ the function  $D_u$ has a non-isolated singularity: The singular set is $X$.
\end{example}

\

\begin{example}[where the line component $\Delta_{\xi}^{\atyp}$ coincides with $\Delta_p^{\iflex}$]\label{ex:regatyp}\ \\
Let 
$$X=\{(x,y)\in\bC^2:xy^4=iy^5+y^3-3y^2+3y-1\}.$$ 
  We have $\EDdeg(X) = 10$.  We will first find $\Delta^{\atyp}$. Let us observe that 
$\bar X \cap H^\ity = \{[1;0], [i;1]\}$, and that $Q^{\ity}\cap H^\ity = \{[i;1]\}$. Thus in order to find $\Delta^{\atyp}$ we have to focus at the point $\xi =[i;1]$ only.
A local parametrization of $X$ at $\xi$, which is actually global, is given by: 
$$t\in\bC^*\to (x(t),y(t));\ \ \  x(t)=\frac{i+t^2-3t^3+3t^4-t^5}{t},\ \  y(t)=\frac 1t.$$
By our study of the structure of  $\Delta^{\atyp}$ in \S\ref{s:struct},  we get: $k=1$, $P(t)=i+t^2-3t^3+3t^4-t^5$ and $Q(t)=1$.
Thus: 
$$h(t,u)= (iu_1+u_2)t +   t^3(-3i-u_1)) +  t^4 (1+6i + 6 u_1) + t^5( -9 - 3i-9 u_1) + \hot$$
 and therefore $\Delta_{\xi}^{\atyp}$ is the line $L :=\{iu_1+u_2=0\}$.  By Theorem \ref{t:morseinfty} and Definition \ref{t:morseinfty}, the generic Morse number at infinity of is then $m^{\gen}_{[i;1]} = 3-1 = 2$. In the focal point $q =(-3i,-3)$ the Morse number at infinity is $m_{[i;1]}(q) = 4-1 = 3$. 
 
 \
 
We claim that the inclusion $\Delta_{\xi}^{\atyp}\subset \Delta^{\reg}$ holds. To prove it, we will use Theorem \ref{t:reg} at 
the point $p=(i,1)$ and the same global parametrization, thus at the value $t_0=1$.
We have:
$$x'(t)=-\frac{i}{t^2}+1-6t+9t^2-4t^3,\ \ \ 
x''(t)=\frac{2i}{t^3}-6+18t-12t^2,$$
$$y'(t)=-\frac{1}{t^2},\ \ \ y''(t)=\frac{2}{t^3},$$
and for $t_{0}=1$, we get:
$$x(1)=i,\ \ \  y(1)=1,\ \ \ x'(1)=-i, \ \ \ x''(1)=2i,\ \ \ y'(1)=-1, \ \ \ y''(1)=2,$$
and 
$$x'(1)y''(1)-y'(1)x''(1)=(-i)2-(-1)(2i)=0,$$
$$(x'(1))^2+(y'(1))^2=(-i)^2+(-1)^2=0.$$

By Theorem \ref{t:reg}(c), every point $(u_1,u_2)$ which satisfies the equation
$$x'(1)(x(1)-u_1)+y'(1)(y(1)-u_2)=0$$ 
is in $\Delta^{\reg}$.  In our case we have:
 $$x'(1)(x(1)-u_1)+y'(1)(y(1)-u_2)=(-i)(i-u_1)+(-1)(1-u_2)=iu_1+u_2,$$
thus our claim is proved.   In fact we have shown, that $\Delta_{\xi}^{atyp} = \Delta_{p}^{\iflex}$. Note that these discriminant lines are related to two different points, namely $\xi$ at infinity and the flex point $p \in X$. Due to the special construction of this both lines coincide in the $u_1u_2$ space.   
\end{example}

\subsection{Isotropic coordinates}\label{ss:isotropic}
The examples with atypical discriminant do not occur in the real setting. Indeed,  the isotropic points 
at infinity $Q^{\ity}$ are not real, and the atypical discriminant is not real either (since consist of lines parallel to the isotropic lines).
We obtain real coefficients when we use ``isotropic coordinates'', as follows: $z:=x+iy$, $w:= x-iy$.
 The data points get coordinates: $v_{1} := u_{1}+i u_{2}$, $v_{2} := u_{1}-i u_{2}$, and the Euclidean distance function takes the following hyperbolic shape: 
 \[ D_v(z,w) =  (z-v_1)(w-v_2). \]
In isotropic coordinates, $Q^{\ity}$ reads $\{zw=0\}$, thus two points: $[0;1]$ and $[1;0]$.
In order to study what happens with the Morse points in the neighborhood of these points at infinity $[0;1]$ and $[1;0]$, we need to change the variables in the formulas of \S\ref{s:struct}. So we recall and adapt as follows:

Let $\gamma:D\to \Gamma$ be a local holomorphic  parametrization of $\Gamma$ at $\xi\in Q^{\ity}$, where 
$D$ is some small enough disk in $\bC$ centred at $0$, and $\gamma(0)=\xi$. For $t\in D$, we write $z_{1}(t):=z_{1}(\gamma(t))$ and $z_{2}(t):=z_{2}(\gamma(t))$, where $z_{1}(t)$ and $z_{2}(t)$ are 
meromorphic on $D$. Then there exists a unique positive integer $k$ such that:
$$z_{1}(t)=\frac{P(t)}{t^k}, \ z_{2}(t)=\frac{Q(t)}{t^k},$$
and $P(t)$ and $Q(t)$ are holomorphic on $D$, where $P(0)$ and $Q(0)$ are not both equal to zero. Note that under these notations we have  $\xi =[P(0);Q(0)]\in H^{\ity}$.

For $t\in D\m\{0\}$ and $v=(v_1,v_2)\in\bC^2$,  we then have the equivalence $((z_{1}(t),z_{2}(t)),v)\in \cE_X \Longleftrightarrow h(t,v) =0$, where:
\begin{equation}\label{eq:test}
 h(t,v)=(tP'(t)-kP(t))(Q(t)-v_2t^k)+(tQ'(t)-kQ(t))(P(t)-v_1t^k),
\end{equation}
and note that $h:D\times\bC^2\to \bC$ is a holomorphic function. 


\begin{example}\label{e:isotropic}
$X = \{z_1^2 z_2 - z_1 = 1 \}$ in isotropic coordinates.  Then 
$X^{\infty}\cap Q^{\ity} = Q^{\ity}$  consist of the two isotropic points, $[0;1]$ and $[1;0]$. On computes  that $\EDdeg(X) =3$.

At the point $[0;1]\in Q^{\ity}$, the curve $\bar X$ has a single local branch, which we denote by $X_{[0;1]}$. We use the parametrization:
 $z_1 = t = \frac{t^3}{t^2}$,  $z_2 = \frac{1+t}{t^2}$. \\
 Thus $k=2$ and  $P = t^3$,   $Q= 1+t$, for $t \to 0$

The condition \eqref{eq:test} becomes: $h(t,v) = 2v_1 t^2 + (v_1-1) t^3 - v_2 t^5$.

In the notations of \S\ref{s:struct}, we get:  $h_0=h_1=0$, and therefore, by Theorem \ref{t:atyp}, we have that  $\Delta^{\atyp}(X_{[0;1]}) = \{v_1 = 0\}$. By Theorem \ref{t:morseinfty}, the Morse number is $m_{X_{[0;1]}}^{\gen}= 3-2=1$ at every point of this line.

\begin{figure}[hbt!]
 \includegraphics[width=65mm]{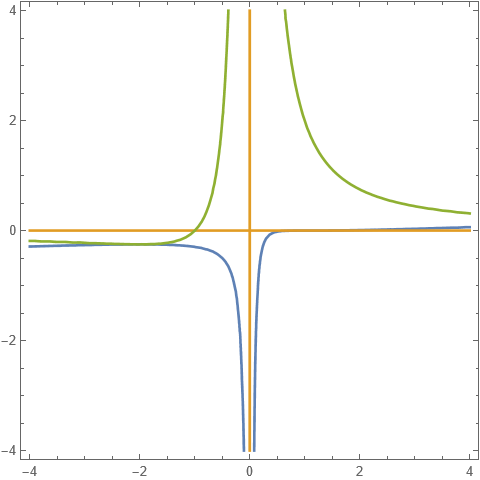}
 \caption{In blue $\Delta^{\reg}$; in brown $\Delta^{\atyp}$; in green a real picture of $X$.}
\end{figure}

At the other isotropic point $[1;0]$, we also have a single branch of $X $, which we denote by $X_{[1;0]}$.
Using the  parametrization $z_1 =  \frac{1}{t}$, $z_2 = \frac{t^3 + t^2}{t}$, we get $k=1$, and  $P =1$, $Q= t^2 + t^3$.
This yields:  $h(t,v) = v_2 t + (1-v_1) t^3 - 2 v_1 t^4$.

Since $h_{0}=0$, by Theorem \ref{t:atyp}, we have that $\Delta^{\atyp}(X_{[1;0]}) = \{v_2 = 0\}$. By Theorem \ref{t:morseinfty}, the Morse number is $m_{X_{[1;0]}}^{\gen} =3-1=2$ at all points of this line except of the point of intersection $(1,0) =\{v_2 = 0 = 1-v_1\}$, where the Morse number is $m_{X_{[0;1]}}((1,0)) =4-1=3$.  At this exceptional point, all the 3 Morse points abut to infinity at $[1;0]$.

Here are some more conclusions:
\begin{itemize}
\item[$\bullet$] $\Delta^{\atyp} = \{ v_1 v_2 = 0\}$: two lines trough the isotropic points,
\item[$\bullet$] $\Delta^{\reg}  = \{  -  v_1^3 +27 v_1^2v_2 + 3v_1^2 -3 v_1 +1 = 0 \}$, as computed with Mathematica \cite{Wo}.
\item[$\bullet$] $\Delta^{\sing}=\emptyset$.
\end{itemize}
Notice that $\Delta^{\atyp}\cap \Delta^{\reg} \ni (1,0)$.  
We have seen above that when moving the data point from outside the discriminant $\Delta_{\tED}$ to the point $(1,0)$,  the 3 Morse points go to infinity at $[1;0]$.
In case we move the data point inside  $\Delta^{\reg}$, then both the Morse point and the non-Morse singular point
go to infinity. Moving the data point along $\{v_2=0\}$, the single Morse point goes to infinity at $[1;0]$. 
\end{example}

\begin{example}{\bf Singular points with isotropic tangents}.\label{e:isotrop}
Recall that $[1;0]$ and $[0;1]$ are the isotropic points.  We can assume that the branch $\Gamma$ at the point $p= (0,0)$ is given by:
$$ z_1 = t^{\alpha}  \; ; \; z_2 = c t^{\beta}  + \hot  \; \;  c \ne 0 \; \;  \beta > \alpha \ge 1  $$
The distance functie $D_v$ gives the following normality condition:
$$ -  \alpha v_2 t ^{\alpha -1} - c \beta v_1 t^{\beta-1} + c (\alpha+\beta) t^{\alpha +\beta -1} + \hot = 0  $$
The condition $v_2 = 0$ defines $\Delta^{\sing}$ and is the normal, but also the tangent line at $p$ !. On $v_2=0$ we get (after dividing bij $t^{\alpha-1})$: 
 $$  - c \beta v_1 t^{\beta - \alpha} + c (\alpha+\beta) t^{\beta} + \hot = 0  $$
If moreover $v_1 \ne 0$ we get $m_p(\Gamma)(v) = \beta - \alpha$; if $v_1=0$  we get $m_p(\Gamma)(\tilde{v})  =\beta$.

The focal discriminant can be computed easily; it cuts $\Gamma$ in $p$. One finds (up to constants and higher order terms):  $ v_1 = t^{\alpha} \; ; \; v_2 =t^{\beta}$;
which gives $\Delta^{\focal}$ the the same cusp type as the original branch $\Gamma$.  This is not a surprise since tangent lines and normal lines  are the same and the envelope of normals and tangents show a similar behavior.

\end{example}


\end{document}